\documentclass[12pt]{amsart}
\author{Kevin Ford}
\author{Paul Pollack}
\address{Department of Mathematics\\1409 West Green Street\\University of Illinois at Urbana-Champaign\\ Urbana, Illinois 61801\\USA}
\usepackage{amsmath,amssymb,amsthm}
\usepackage[margin=1in]{geometry}
\newtheorem*{conjUHL}{Hypothesis UL}

\newtheorem{thm}{Theorem}

\newtheorem{lem}{Lemma}
\theoremstyle{remark}

\DeclareMathAlphabet{\curly}{U}{rsfs}{m}{n}

\newcommand{\be}{\begin{equation}}
\newcommand{\ee}{\end{equation}}
\newcommand{\benn}{\begin{equation*}}
\newcommand{\eenn}{\end{equation*}}
\newcommand{\bal}{\begin{align*}}
\newcommand{\ea}{\end{align*}}
\newcommand{\eal}{\ensuremath{\end{align*}}}
\newcommand{\bea}{\begin{eqnarray}}
\newcommand{\eea}{\end{eqnarray}}

\renewcommand{\a}{\ensuremath{\alpha}}

\newcommand{\eps}{\ensuremath{\varepsilon}}
\renewcommand{\(}{\left(}
\renewcommand{\)}{\right)}

\renewcommand{\le}{\leqslant}
\renewcommand{\leq}{\leqslant}

\renewcommand{\geq}{\geqslant}

\begin{document}
\title{On common values of $\phi(n)$ and $\sigma(m)$, I}
\begin{abstract}
We show, conditional on a uniform version of the prime $k$-tuples conjecture, that there are $x/(\log{x})^{1+o(1)}$ numbers not exceeding $x$ common to the ranges of $\phi$ and $\sigma$.  Here $\phi$ is Euler's
totient function and $\sigma$ is the sum-of-divisors function.
\end{abstract}

\thanks{
The first author was supported by NSF Grant DMS-0901339. 
The second author was supported by an NSF Postdoctoral Fellowship (award DMS-0802970).  The research was conducted in part 
while the authors were visiting the
Institute for Advanced Study, the first author supported by grants
from the Ellentuck Fund and The Friends of the Institute For
Advanced Study.  Both authors thank the IAS for its
hospitality and excellent working conditions.}

\renewcommand{\labelenumi}{(\roman{enumi})}
\def\d{\delta}
\def\D{\mathcal{D}}
\def\A{\mathcal{A}}
\def\B{\mathcal{B}}
\def\fancyA{\curly{A}}
\def\fancyB{\curly{B}}
\def\I{\mathcal{I}}
\def\J{\mathcal{J}}
\def\V{\curly{V}}
\def\N{\mathbf{N}}
\def\Q{\mathbf{Q}}
\def\Z{\mathbf{Z}}
\def\Cc{\mathcal{C}}
\def\Pp{\mathcal{P}}
\def\Ss{\mathcal{S}}
\maketitle


\section{Introduction}
For each positive-integer valued arithmetic function $f$, let $\V_{f}
 \subset \N$ denote the image of $f$, and put $\V_{f}(x) := \V_f\cap [1,x]$
 and $V_f(x) := \#\V_f(x)$. In this paper we are primarily concerned with
 the cases when $f=\phi$, the Euler totient function, and when $f=\sigma$,
 the usual sum-of-divisors function. When $f=\phi$, the study of the
 counting function $V_f$ goes back to Pillai \cite{pillai29}, and was
 subsequently taken up by Erd\H{o}s \cite{erdos35, erdos45}, Erd\H{o}s and
 Hall \cite{EH73, EH76},  Pomerance \cite{pomerance86}, Maier and Pomerance
 \cite{MP88}, and Ford \cite{ford98} (with an announcement in
 \cite{ford98A}). From the sequence of results obtained by these authors,
we mention Erd\H{o}s's asymptotic formula (from \cite{erdos35}) for
 $\log\frac{V_f(x)}{x}$, namely
\begin{equation}\label{eq:erdosestimate0}
 V_f(x) = \frac{x}{(\log{x})^{1+o(1)}} \quad (x\to\infty) \end{equation}
and the much more intricate determination of the precise order of magnitude by Ford,
\begin{equation}\label{eq:fordestimate0} V_f(x) \asymp \frac{x}{\log{x}} \exp(C(\log_3{x}-\log_4{x})^2+D\log_3{x}-(D+1/2-2C)\log_4{x}). \end{equation}
Here the constants $C$ and $D$ are defined as follows: Let
\begin{equation}\label{eq:somedefs} F(z) :=\sum_{n=1}^{\infty} a_n z^n, \quad\text{where}\quad a_n = (n+1)\log(n+1)-n\log{n}-1.\end{equation}
Since each $a_n >0$ and $a_n \sim \log{n}$ as $n\to\infty$, it follows that $F(z)$ converges
to a continuous, strictly increasing function on $(0,1)$, and $F(z)\to\infty$ as $z\uparrow 1$. Thus there is a unique real number $\rho$ for which
\begin{equation}\label{eq:rhodef} F(\rho) = 1 \quad (\rho = 0.542598586098471021959\ldots).\end{equation}
In addition, $F'$ is strictly increasing, and
$F'(\rho) = 5.697758 \ldots$.
Then $C = \frac{1}{2|\log \rho|} = 0.817814\ldots$ and
$D = 2C(1+\log F'(\rho) - \log(2C)) - 3/2 = 2.176968 \ldots$.
In \cite{ford98}, it is also shown that \eqref{eq:fordestimate0} holds for a wide class of $\phi$-like functions, including $f=\sigma$. Consequently,  $V_\phi(x) \asymp V_{\sigma}(x)$.

Erd\H{o}s (see \cite[8, p. 172]{erdos59} or \cite{EG80}) asked if it could be proved that infinitely many natural numbers appear in both $\V_{\phi}$ and $\V_{\sigma}$. This question was recently answered by Ford, Luca, and Pomerance \cite{FLP10}. Writing $V_{\phi,\sigma}(x)$ for the number of common values of $\V_\phi$ and $\V_\sigma$ up to $x$, they proved that
\[ V_{\phi,\sigma}(x) \geq \exp((\log\log{x})^{c}) \]
for some positive constant $c > 0$ and all large $x$
(in \cite{Ga} this is shown for \emph{all} constants $c>0$).
This lower bound is probably very far from the truth. A naive guess, based on \eqref{eq:erdosestimate0} and the
hypothesis of independence, might be that $V_{\phi,\sigma}(x) = x/(\log{x})^{2+o(1)}$. However, the analysis of \cite{ford98} indicates that elements of $\V_{\phi}$ and $\V_{\sigma}$ share many structural features, which suggests that perhaps $V_{\phi,\sigma}$ is larger than this naive prediction.

In this paper we show that this is indeed the case, subject to the following plausible hypothesis:
\begin{conjUHL} Suppose $a_1, \dots, a_h$ are positive integers and $b_1, \dots, b_h$ are integers such that $\prod_{1 \leq i < j \leq h}(a_i b_j - a_j b_i) \neq 0$. Assume that for some constant $A > 0$, we have
\[ \max_{1 \leq i \leq h}\{|a_i|, |b_i|\} \leq x^{A}. \]
Then for large $x$, depending on $A$ and $h$, the number of natural numbers $n\leq x$ for which $a_i n + b_i$ is prime for every $1 \leq i \leq h$ is
\[ \gg_{A, h} C \frac{x}{(\log{x})^h}. \]
Here $C$ is the singular series associated to $\{a_i n + b_i\}_{i=1}^{h}$, defined by
\[ C := \prod_{p}\frac{1-\nu(p)/p}{(1-1/p)^{h}}, \quad\text{where}\quad\nu(p):=\#\{n\bmod{p}: \prod_{i=1}^{h}{(a_i n + b_i)}\equiv 0 \pmod{p}\}.  \]
\end{conjUHL}

This hypothesis is a quantitative form of Dickson's prime $k$-tuples conjecture. The name ``Hypothesis UL'' (with ``$L$'' for linear) is suggested by an analogous hypothesis proposed by Martin \cite{martin02} to study smooth values of polynomials. His ``Hypothesis UH'' makes a somewhat stronger prediction in the more general context of Hypothesis H, in a similar range of uniformity.  A very special case of Hypothesis
UL, that the number of twin primes $p,p+2 \le x$ is $\gg x/\log^2 x$,
implies immediately that $ V_{\phi,\sigma}(x) \gg x/\log^2 x$.

\begin{thm}\label{thm:phisiglower} Assume Hypothesis UL. Then as $x\to\infty$,
\[ V_{\phi,\sigma}(x) \geq \frac{x}{(\log{x})^{1+o(1)}}. \]
\end{thm}

The proof, which proceeds along entirely different lines than \cite{FLP10}, has its origin in the following simple observation: Write $R_{f}(v):= \#f^{-1}(v)$ for the number of preimages of $v$ under the arithmetic function $f$. By H\"{o}lder's inequality, we have
\begin{equation}\label{eq:holder}\left(\sum_{v \leq x}R_\phi(v) R_\sigma(v)\right)^3 \leq V_{\phi,\sigma}(x) \left(\sum_{v \leq x} R_{\phi}(v)^2 R_{\sigma}(v)\right) \left(\sum_{v \leq x} R_{\phi}(v) R_{\sigma}(v)^2\right). \end{equation}
In particular, to prove Theorem \ref{thm:phisiglower}, it would suffice to show that the left-hand sum is bounded below by $x/(\log{x})^{1+o(1)}$ while the two sums appearing on the right-hand side are bounded above by $x/(\log{x})^{1+o(1)}$. Unfortunately these estimates are not so easy to obtain. It turns out that rather than count all preimages, as in our definition of $R_f$ above, it is easier to obtain analogous estimates if we count only preimages belonging to certain specially constructed sets. The choice of these sets is motivated by the detailed structure theory of preimages developed in \cite{MP88} and \cite{ford98}.

\subsection*{Notation}  Most of our number-theoretic notation is standard. Possible exceptions include
$P^+(n)$ for the largest prime factor of $n$, and $\Omega(n,U,T)$ for the total number of prime factors $p$ of $n$
with $U < p \le T$, counted according to multiplicity.

Big-Oh notation and the related symbols ``$\ll$,'' ``$\gg$,'' and ``$\asymp$'' appear with their usual meanings,
including subscripts to indicate the dependence of implied constants. We use $o_k(1)$ for a quantity that tends to zero for each fixed value of $k$.
We also put $\log_1{x} = \max\{1,\log{x}\}$ and we write $\log_k$ for the $k$th iterate of $\log_1$.


\section{Proof of Theorem \ref{thm:phisiglower}}
We now construct our surrogate representation functions. For a set
$\fancyB$ of natural numbers and $f$ an arithmetic function, let
\[ R_{f}(v;\fancyB) := \#\{n \in \fancyB: f(n)=v\}. \]
Then \eqref{eq:holder} continues to hold if we replace $R_\phi(v)$ by $R_{\phi}(v; \fancyB_{\phi})$ and $R_{\sigma}(v)$ by $R_{\sigma}(v; \fancyB_{\sigma})$. We now describe our choices of $\fancyB_{\phi}$ and $\fancyB_{\sigma}$.

It is convenient to work not with a single set $\fancyB_{\phi}$, but with a family of such sets, and similarly for $\fancyB_{\sigma}$. Our definition of these sets depends on a real parameter $\a$, which we always suppose satisfies $1/2 < \a < \rho$ (with $\rho$ as in \eqref{eq:rhodef}), on a natural number parameter $k$, and on $x$. We define $\fancyB_{\phi}^{\alpha,k}(x)$ as the set of natural numbers $n$ possessing all of the following properties:
\begin{enumerate}
\item $n$ is the product of $k$ distinct primes and $\phi(n) \leq x$.
\item If $p_0 > \cdots > p_{k-1}$ is the decreasing list of the primes dividing $n$, then
\[ v_i^{1/12} < P^{+}(p_i-1) \leq p_i-1 \leq v_i, \quad\text{and}\quad v_i = \exp((\log{x})^{\alpha^i});\]
also, $P^{+}(p_i-1)$ is the unique prime divisor of $p_i-1$ exceeding $v_i^{1/12}$ for $1 \leq i \leq k-1$.
\item If $1 \leq j \leq i \leq k$, we have
\[ \left|\Omega(p_{j-1}-1,v_i,v_{i-1})-(\a^{i-1}-\a^{i})\log_2{x}\right| \leq 2k\sqrt{(\a^{i-1}-\a^i)\log_2{x}}.\]
\item $6$ is the largest factor of $p_i-1$ supported on the primes $\leq v_k$.
\item If $p \mid \phi(n)$ and $p > v_k$, then $p \parallel \phi(n)$.
\end{enumerate}
We define $\fancyB_{\sigma}^{\alpha,k}(x)$ analogously, with $\phi$ replaced by $\sigma$ in (i) and (v) and $p-1$ replaced by $p+1$ throughout in
(ii)--(iv).
If $\alpha$, $k$, and $x$ are all understood, we write simply $\fancyB_{\phi}$ and $\fancyB_{\sigma}$.

In order to establish Theorem \ref{thm:phisiglower},
it is enough to prove the following two estimates:
\begin{lem}\label{lem:lower} Assume Hypothesis UL. Let $\epsilon > 0$. There is a real number $1/2 < \alpha_0 < \rho$ and a natural number $k_0$ with the following property: If $\alpha_0 < \alpha < \rho$ and $k \geq k_0$, then for all large enough $x$ (depending on $\alpha$ and $k$),
\[ \sum_{v \leq x} R_{\phi}(v; \fancyB_{\phi})R_{\sigma}(v; \fancyB_{\sigma}) \geq \frac{x}{(\log{x})^{1+\epsilon}}.\]
In other words,
there are at least $x/(\log{x})^{1+\epsilon}$ solutions $(n,m)$ to
\[ \phi(n)=\sigma(m), \quad\text{where}\quad (n,m)\in \fancyB_{\phi}\times \fancyB_{\sigma}. \]
\end{lem}

\renewcommand{\labelenumi}{(\alph{enumi})}

\begin{lem}\label{lem:upper} Let $\epsilon > 0$. There is a natural number $k_0$ with the following property: If $1/2 < \a < \rho$ and $k \geq k_0$, then for all large enough $x$ (depending on $\alpha$ and $k$),
\[ \sum_{v \leq x} R_{\phi}(v; \fancyB_{\phi})^2 R_{\sigma}(v; \fancyB_{\sigma}) \leq \frac{x}{(\log{x})^{1-\eps}}.\]
In other words,
there are at most $x/(\log{x})^{1-\eps}$ solutions $(n,n',m)$ to
\[ \phi(n)=\phi(n')=\sigma(m), \quad\text{where}\quad (n,n',m)\in \fancyB_{\phi}\times\fancyB_{\phi}\times \fancyB_{\sigma}. \]
The same bound holds for $\sum_{v \leq x} R_{\phi}(v; \fancyB_{\phi}) R_{\sigma}(v; \fancyB_{\sigma})^2$.
\end{lem}

Note that Hypothesis UL is required for the proof of Lemma \ref{lem:lower}, while Lemma \ref{lem:upper} is unconditional.

\subsection{Technical preliminaries} We collect some technical results that will be used in the proofs of Lemmas \ref{lem:lower} and \ref{lem:upper}. The first concerns the distribution of prime factors in a `typical' factorization of a squarefree number $N$.

\begin{lem}\label{lem:normalfact} Let $N$ be a squarefree natural number with $I$ prime factors. Consider all $i^{I}$ ways of writing $N$ as a product of $i$ natural numbers, say $N = d_1 \cdots d_i,$ where the order of the factors is taken into account. For any $\Delta>0$, the number of such decompositions with
\[ |\omega(d_1) - I/i| \geq \Delta \sqrt{I/i} \]
is at most $\Delta^{-2} i^{I}$, uniformly for $\Delta > 0$.
\end{lem}

\begin{proof} Let $\mathbf{X} = (X_1, \dots, X_i)$, where each $X_i = \omega(d_i)$. Viewing $\mathbf{X}$ as a random vector defined on the space of all decompositions of $N$ into $i$ factors, observe that $\mathbf{X}$
 follows a multinomial distribution. The lemma now follows from Chebyshev's inequality, taking into account that $\textbf{E}[X_1]=I/i$ and $\mathrm{var}(X_1) = (I/i)(1-1/i) \leq I/i$.
\end{proof}

The following estimate is well-known from the study of sieve methods (see, e.g., \cite[Theorem 4.2]{HR74}).
\begin{lem}\label{lem:sieve} Suppose $a_1, \ldots, a_h$ are positive integers and
$b_1, \ldots, b_h$ are integers such that
$$
E := \prod_{i=1}^h a_i \prod_{1\le i<j\le h} (a_ib_j-a_jb_i) \ne 0.
$$
Then
\[
\#\{ n\le x: a_in+b_i \text{ prime } (1\le i\le h) \}
\ll_h \frac{x}{(\log x)^h} \prod_{p}
\frac{1-\nu(p)/p}{(1-1/p)^h}
\ll_h \frac{x(\log_2 (E+2))^h}{(\log x)^{h}},
\]
where $\nu(p)$ is the number of solutions of the congruence
$\prod (a_i n+b_i)\equiv 0\pmod{p}$, and
the implied constant may depend on $h$.\end{lem}

The next two lemmas concern the Poisson distribution.
\begin{lem}\label{lem:poisson2} If $z> 0$ and $\Delta >0$, then
\[ \sum_{|k-z| > \Delta z} \frac{z^k}{k!} \leq \Delta^{-2} e^z. \]
\end{lem}
\begin{proof} This follows immediately from Chebyshev's inequality, once we recall that the Poisson distribution with parameter $z$ has mean and variance both equal to $z$.
\end{proof}

\begin{lem}[see e.g. {\cite[Lemma 2.1]{ford98}}]\label{lem:poisson1} If $z > 0$ and $0 < \alpha < 1 < \beta$, then
\[ \sum_{k \leq \alpha z} \frac{z^k}{k!} < \left(\frac{e}{\alpha}\right)^{\alpha z} \quad \text{and}\quad \sum_{k \geq \beta z} \frac{z^k}{k!} < \left(\frac{e}{\beta}\right)^{\beta z}. \]
\end{lem}

\subsection{Proof of Lemma \ref{lem:lower}}  Suppose that $n = p_0 \cdots p_{k-1}$ and $m = q_0 \cdots q_{k-1}$ are squarefree numbers satisfying $\phi(n)=\sigma(m) \in [1,x]$, where the primes are ordered so that
\[ p_0 > p_1 > \dots > p_{k-1} \quad \text{and}\quad q_0 > q_1 >\dots > q_{k-1}. \]
Then
\begin{equation}\label{eq:coincidence}
 (p_0-1)(p_1-1)\cdots(p_{k-1}-1) = (q_0+1)(q_1+1)\cdots(q_{k-1}+1).\end{equation}
We consider separately the prime factors of each shifted prime lying in
each interval $(v_i,v_{i-1}]$, where $v_i = \exp((\log{x})^{\alpha^i})$. For $0\le j \le k-1$ and $0\le i\le k$, let
\[
s_{i,j} := \prod_{\substack{ p^a\parallel (p_j-1) \\p\le v_i }} p^a, \qquad
s'_{i,j} := \prod_{\substack{ p^a\parallel (q_j+1) \\p\le v_i }} p^a, \qquad
s_i := \prod_{j=0}^{k-1} s_{i,j} = \prod_{j=0}^{k-1} s'_{i,j}.
\]
Also, for $0\le j\le k-1$ and $1\le i\le k$, let
\[
t_{i,j} := \frac{s_{i-1,j}}{s_{i,j}}, \qquad t'_{i,j} := \frac{s'_{i-1,j}}{s'_{i,j}},
\qquad t_i := \prod_{j=0}^{k-1} t_{i,j} =\prod_{j=0}^{k-1} t'_{i,j}.
\]
Let
\begin{align}
\sigma_i &= \{s_i;s_{i,0},\ldots,
s_{i,k-1}; s_{i,0}', \ldots,
s_{i,k-1}'\},\label{eq:sigmaiform}\\
\tau_i &= \{t_i; t_{i,0},\ldots,t_{i,k-1};
t'_{i,0},\ldots,t'_{i,k-1}\}.\label{eq:tauiform}
\end{align}
Observe that if we define multiplication of $(2k+1)$-tuples component-wise, then
we have
\be\label{sigtau}
\sigma_{i-1} = \sigma_i \tau_i.
\ee

Suppose we are given a collection of squarefree solutions $(n,m)$ to \eqref{eq:coincidence} for which $m,n\le x$. Let $\mathfrak S_i$ denote the set of $\sigma_i$ that arise from these solutions, and let $\mathfrak T_i$ denote the corresponding set of $\tau_i$. For $1 \leq i \leq k$, let
\[ \mathfrak U_i:= \{(\sigma,\tau): \sigma \in \mathfrak S_i, \tau \in \mathfrak T_i, \sigma \tau \in \mathfrak S_{i-1}\}. \]
The given set of solutions $(n,m)$ is in one-to-one correspondence with the set $\mathfrak S_0$, since
\[ \sigma_0 = (\phi(n); p_0-1,\dots,p_{k-1}-1, q_0+1, q_1+1, \dots, q_{k-1}+1) \]
both determines the pair $(n, m)$ and is determined by it. Also, from \eqref{sigtau} we see that the set $\mathfrak S_0$ is completely determined once we know $\mathfrak S_k$ and each of the sets $\mathfrak U_k, \mathfrak U_{k-1}, \dots, \mathfrak U_{1}$. To construct a set of solutions, we can reverse the process, first picking a set $\mathfrak{S}_k$ and then successively constructing $\mathfrak{U}_k, \dots, \mathfrak{U}_1$. We carry out this plan, verifying that $(n,m) \in \fancyB_{\phi}\times \fancyB_{\sigma}$ for all the solutions constructed in this way.

%
%
%

We begin by putting $\mathfrak S_k:=\{\sigma\}$, where $\sigma:= (6^k; 6, \dots, 6; 6, \dots, 6)$.

Suppose that $\mathfrak S_i$ has been determined, where $2 \leq i \leq k$. For each $\sigma_i \in \mathfrak S_i$, write $\sigma_i$ in the form \eqref{eq:sigmaiform}. As part of the induction hypothesis, suppose that each $\sigma_i$ satisfies
\begin{equation}\label{eq:wellformed1} s_i = \prod_{j=0}^{k-1}{s_{i,j}} = \prod_{j=0}^{k-1}{s_{i,j}'}, \end{equation}
\begin{equation}\label{eq:wellformed2}\max_{0 \leq j \leq k-1}\{s_{i,j}, s_{i,j}'\} \leq v_i. \end{equation}
Moreover, suppose also that for $j=i, i+1, \dots, k-1$, we have
\begin{equation}\label{eq:twoprimes}
 p_j:= s_{j,j}+1 \quad\text{and} \quad q_j := s_{j,j}-1\end{equation}
all prime. Clearly all of these hypotheses hold when $i=k$ (the last condition being vacuous).

Now we construct $\mathfrak U_i$ and so determine $\mathfrak S_{i-1}$. Let  $t_i^{\ast}$ range over all numbers satisfying
\begin{enumerate}
\item $t_i^{\ast}$ is squarefree,
\item every prime divisor of $t_i^{\ast}$ belongs to $(v_i, v_{i-1}^{\frac{1}{12\log_2{x}}}]$,
\item $t_i^{\ast}$ has exactly $N_i:= \lfloor i(\a^{i-1}-\a^i)\log_2{x}\rfloor$ prime divisors,
\end{enumerate}
and suppose that the variables $t_{i,0}, \dots, t_{i, i-2}, t_{i,0}', \dots, t_{i,i-2}', u_i, u_i'$ range over all (ordered) dual factorizations of $t_i^{\ast}$ of the shape
\begin{equation}\label{eq:dualfact}
 t_i^{\ast} = t_{i,0} \dots t_{i,i-2} u_i = t_{i,0}' \cdots t_{i,i-2}' u_i'
\end{equation}
for which each of the variables $t_{i,j},t'_{i,j},u_i,u_i'$ satisfies
\begin{equation}\label{eq:tiprimes}
 |\Omega(\cdot, v_i, v_{i-1}) - (\a^{i-1}-\a^i)\log_2{x}| < k \sqrt{(\a^{i-1}-\a^i)\log_2{x}}. \end{equation}
Let $Q_i$ range over all primes in the interval
\begin{equation}\label{eq:Qdefs0} v_{i-1}^{1/12} < Q_i \leq v_{i-1}^{1/6} \end{equation}
for which also
\begin{equation}\label{eq:Qdefs}
p_{i-1}:=s_{i,i-1} u_i Q_i +1\quad \text{and} \quad q_{i-1}:=s_{i,i-1}'u_i' Q_i -1
\end{equation}
are prime. We put $t_{i,i-1} := u_iQ_i$, $t_{i,i-1}' := u_i'Q_i$, and $t_i:=t_i^{\ast}Q_i$ (so that $t_i=\prod_{j}{t_{i,j}}=\prod_{j}{t_{i,j}'}$) and we add to $\mathfrak U_{i}$ all pairs of the form $(\sigma_i, \tau_i)$, where
\[ \tau_i:= (t_i; t_{i,0}, \dots, t_{i,i-1}, 1, \dots, 1; t_{i,0}', \dots, t_{i,i-1}', 1, \dots, 1). \]
For the set $\mathfrak{S}_{i-1} = \{\sigma_i \tau_i: (\sigma_i, \tau_i) \in \mathfrak U_i\}$ determined this way, our induction hypotheses
\eqref{eq:wellformed1}-\eqref{eq:twoprimes} continue to hold. Indeed, \eqref{eq:wellformed1} and
\eqref{eq:twoprimes} hold by construction. To verify \eqref{eq:wellformed2} for $i-1$ in place of $i$, observe that if $j \neq i-1$, then
\[ s_{i-1,j} = s_{i,j} t_{i,j} \leq v_i \left((v_{i-1})^\frac{1}{12 \log_2{x}}\right)^{\Omega(t_{i,j})}<  v_i v_{i-1}^{1/6} < v_{i-1} \]
for large $x$, by our induction hypothesis and the inequality $\Omega(t_{i,j}) \leq N_i \leq 2\log_2{x}$.
(Throughout this proof, the meaning of ``large'' $x$ is allowed to depend on $\alpha$ and $k$.)
If $j=i-1$, then
\[ s_{i-1,j} = s_{i,j} t_{i,j} \leq v_i u_i Q_i \leq v_i v_{i-1}^{1/6} v_{i-1}^{1/6} < v_{i-1}. \]
If $j > i-1$, then $s_{i-1,j} = s_{i,j}$, and so $s_{i-1,j} \leq v_i \leq v_{i-1}$. Analogous estimates hold
for $s_{i-1,j}'$ in place of $s_{i-1,j}$, giving \eqref{eq:wellformed2}.

At this stage we have determined all of $\mathfrak{S}_k, \dots, \mathfrak{S}_{1}$.
It remains to construct $\mathfrak U_1$ and so determine $\mathfrak S_0$.
Let $\sigma_1 \in \mathfrak S_1$, and write $\sigma_1$ in the form
\[ \sigma_1 = (s_1; s_{1,0}, p_1-1, \dots, p_{k-1}-1; s_{1,0}', q_1+1, \dots, q_{k-1}+1). \]
Let $t'$ range over all natural numbers satisfying
\begin{enumerate}
\item $t'$ is squarefree,
\item $t' \leq x^{1/3}/s_1$,
\item every prime dividing $t'$ belongs to $(v_1, v_0]$,
\item $|\Omega(t',v_1,v_0)-(1-\a)\log_2{x}| \leq k\sqrt{(1-\a)\log_2{x}}$.
\end{enumerate}
For each $t'$, let $Q_1$ range over primes with
\begin{equation}\label{eq:Qdef2} x^{1/2} \leq Q_1 \leq \frac{x}{s_1 t'} \quad\text{for which}\quad p_0:= s_{1,0}t' Q_1+1, \quad q_0:= s_{1,0}'t'Q_1-1 \text{ are prime}.
\end{equation}
We set $t_1 := t'Q_1$, and let $\mathfrak U_1$ consist of all tuples of the form $(\sigma_1, \tau_1)$, where
\[ \tau_1 := (t_1; t_1, 1, 1, \dots, 1; t_1, 1, 1, \dots, 1). \]
Finally, we put $\mathfrak S_0 = \{\sigma_1 \tau_1: (\sigma_1, \tau_1) \in \mathfrak U_1\}$.

The remainder of the proof consists of verifying that the set $\mathfrak S_0$ determined by this construction is as large
as claimed and that the solutions corresponding to the elements of $\mathfrak S_0$ belong to $\fancyB_{\phi} \times\fancyB_{\sigma}$. The lower bound for $\#\mathfrak S_0$ will be made to depend on a lower bound for $\sum_{\sigma_1 \in \mathfrak S_1} 1/s_1$. First, observe that
\begin{equation}\label{eq:sklower} \sum_{\sigma_1 \in \mathfrak S_k} \frac{1}{s_k} = \frac{1}{6^k} \gg_{k} 1. \end{equation}
For $2\leq i \leq k$, we have
\be\label{eq:reciplower}
 \sum_{\sigma_{i-1} \in \mathfrak S_{i-1}} \frac{1}{s_{i-1}} = \sum_{\sigma_i \in \mathfrak S_i}\frac{1}{s_i}\sum_{\tau_i:~(\sigma_i, \tau_i)\in \mathfrak U_i} \frac{1}{t_i} = \sum_{\sigma_i \in \mathfrak
 S_i}\frac{1}{s_i}\sum_{t_i^{\ast}} \frac{h(t_i^{\ast})}{t_i^{\ast}}\sum_{Q_i}\frac{1}{Q_i}.
\ee
Here $t_i^{\ast}$, $t_i$, and $Q_i$ are as the quantities appearing in the description of $\mathfrak{U_i}$ for $2 \leq i \leq k$ (see \eqref{eq:dualfact}--\eqref{eq:Qdefs}), and $h(t_i^{\ast})$ is the number of dual factorizations of $t_i^{\ast}$ of the form \eqref{eq:dualfact}, where the factors $u_i, u_i', t_{i,j}, t_{i,j}'$ all satisfy \eqref{eq:tiprimes}.

By our choice of $\mathfrak S_k$, we have that $6$ divides both $s_{i,i-1}$ and $s_{i,i-1}'$. It follows that the singular series corresponding to the affine linear forms $T, s_{i,i-1} u_i T +1$, and $s_{i,i-1}' u_i' T -1$ is bounded away from zero (using $\nu(p)\le \min(p-1,3)$). Moreover, all the coefficients of these forms are bounded by $v_{i-1}$. Hence Hypothesis UL and partial summation shows that for the inner sum in \eqref{eq:reciplower},
\begin{equation}\label{eq:Qlower} \sum_{Q_i} \frac{1}{Q_i} \gg \frac{1}{\log^2{v_{i-1}^{1/12}}} \gg \frac{1}{(\log{x})^{2\a^{i-1}}}, \end{equation}
where the implied constant is absolute. Inserting this estimate in \eqref{eq:reciplower}, we find that
\begin{equation}\label{eq:reciplower2.0} \sum_{\sigma_{i-1} \in \mathfrak S_{i-1}} \frac{1}{s_{i-1}} \gg \frac{1}{(\log{x})^{2\a^{i-1}}} \sum_{\sigma_i \in \mathfrak S_i}\frac{1}{s_i}\sum_{t_i^{\ast}} \frac{h(t_i^{\ast})}{t_i^{\ast}}. \end{equation}

Recalling the definition of $h(\cdot)$, we see that for each $t_i^{\ast}$,
\begin{align*} h(t_i^{\ast}) &\geq \#\{\text{dual $i$-fold factorizations of $t_i^{\ast}$}\} - \#\{\text{dual $i$-fold factorizations of $t_i^{\ast}$ failing \eqref{eq:tiprimes}}\} \\ &\geq i^{2N_i} - 2i^{N_i}h'(t_i^{\ast}),\end{align*}
where $h'(t_i^{\ast})$ is the number of (single) $i$-fold decompositions of $t_i^{\ast}$ where (at least) one of the factors fails to satisfy \eqref{eq:tiprimes}. By Lemma \ref{lem:normalfact}, we have $h'(t_i^{\ast}) \ll i (i^{N_i}/k^2) \leq i^{N_i}/k$, and thus $h'(t_i^{\ast}) < i^{N_i}/4$, assuming (as we may) that $k$ is sufficiently large. Hence
\begin{equation}\label{eq:hlower}
 h(t_i^{\ast}) \geq \frac{1}{2}i^{2N_i}, \end{equation}
uniformly in $t_i^{\ast}$. Moreover, by the multinomial theorem, if we put
\[ S:= \sum_{v_i < p \leq v_{i-1}^{\frac{1}{12\log_2{x}}}} \frac{1}{p} \quad\text{and}\quad S':=\sum_{v_i < p \leq v_{i-1}}\frac{1}{p^2}, \]
then
\begin{equation}\label{eq:reciplower2}
 \sum_{t_i^{\ast}}\frac{1}{t_i^{\ast}} \geq \frac{S^{N_i}}{N_i!}  - \frac{S^{N_i-2}S'}{(N_i-2)!} \geq \frac{S^{N_i}}{N_i!}\( 1-O(N_i^2 S'/S^2)\) \geq \frac{1}{2} \frac{S^{N_i}}{N_i!}
\end{equation}
for large $x$.
Combining the results of  \eqref{eq:hlower} and \eqref{eq:reciplower2} with \eqref{eq:reciplower2.0}, we find that
\begin{align} \notag\sum_{\sigma_{i-1} \in \mathfrak S_{i-1}} \frac{1}{s_{i-1}} &\gg \frac{1}{(\log{x})^{2\a^i}} \frac{(i^2 S)^{N_i}}{N_i!} \sum_{\sigma_i \in \mathfrak{S}_i}\frac{1}{s_i} \\&\gg \frac{1}{(\log{x})^{2\a^i}} \frac{1}{\sqrt{N_i}} \left(\frac{ei^2 S}{N_i}\right)^{N_i} \sum_{\sigma_i \in \mathfrak{S}_i}\frac{1}{s_i}. \label{eq:lowersim1} \end{align}
(Here we have used Stirling's formula to estimate $N_i!$.) A routine computation, making use of the estimates
\[ S= (\a^{i-1}-\a^i)\log_2{x} + O(\log_3{x}), \quad\quad N_i = i(\a^{i-1}-\a^i)\log_2{x} + O(1), \]
shows that as $x\to\infty$,
\[ \sum_{\sigma_{i-1} \in \mathfrak S_{i-1}} \frac{1}{s_{i-1}} \geq (\log{x})^{(\a^{i-1}-\a^i)(i+i\log{i})-2\a^{i-1} + o_k(1)} \sum_{\sigma_i \in \mathfrak S_i}\frac{1}{s_i}. \]
Starting with \eqref{eq:sklower} and then descending from $i=k$ down to $i=2$, we obtain  that
\[ \sum_{\sigma_1\in \mathfrak S_1} \frac{1}{s_1} \geq (\log{x})^{\sum_{i=2}^{k}\left((\a^{i-1}-\a^i)(i+i\log{i}) -2\a^{i-1}\right) + o_k(1)}. \]

Letting $t'$ and $Q_1$ denote the quantities appearing in the definition of $\mathfrak{U}_1$, we have from Hypothesis UL and the above lower bound on $\sum 1/s_1$ that
\begin{multline}
 \#\mathfrak S_0 = \sum_{\sigma_1 \in \mathfrak S_1} \sum_{\tau_1: (\sigma_1, \tau_1) \in \mathfrak U_1}1
 =\sum_{\sigma_1 \in \mathfrak S_1} \sum_{t'} \sum_{Q_1} 1 \gg \sum_{\sigma_1 \in \mathfrak S_1} \sum_{t'} \frac{x}{s_1 t'\log^3{x}}\\= \frac{x}{\log^3{x}} (\log{x})^{\sum_{i=2}^{k}\left((\a^{i-1}-\a^i)(i+i\log{i}) -2\a^{i-1}\right) + o_k(1)} \sum\frac{1}{t'}.\label{eq:lowers}
\end{multline}
(Note that $\max\{s_{1,0}t',s_{1,0}'t'\} \leq v_1 t' \leq x^{1/3}$ and that $6$ divides both $s_{1,0}$ and $s_{1,0}'$.) We now estimate  $\sum 1/t'$ from below.  Let us temporarily ignore the restriction (d) on $\Omega(t')$, and for brevity write
$T=\{t'\le x^{1/3}/s_1:t' \text{ squarefree}, p|t'\implies
v_1 < p \le v_0 \}$.
Then for large $x$,
\[
\sum_{t'\in T}\frac{1}{t'} \prod_{p \leq v_1}\left(1+\frac{1}{p}\right) \geq \sum_{t' \leq x^{1/3}/s_1, \text{ squarefree}} \frac{1}{t'} \gg \log(x^{1/3}/s_1) \gg \log{x}.
\]
(Note that $s_1 \leq v_1^{k} = x^{o_k(1)}$.) So by Mertens' theorem,
\begin{equation}\label{eq:treciplower1}
\sum_{t'\in T} \frac{1}{t'} \gg \frac{\log{x}}{\log{v_1}} = (\log{x})^{1-\a}.
\end{equation}
To obtain a corresponding lower bound incorporating (d), we show those $t'$ for which (d) is violated make a negligible contribution to \eqref{eq:treciplower1}. Redefine
\[
S:= \sum_{v_1 < p \leq v_0} \frac{1}{p} = (1-\a)\log_2{x}+O(1),
\]
and observe that by Lemma \ref{lem:poisson2}, for large $x$,
\begin{align}\notag
\sum_{\substack{t'\in T \\ |\Omega(t')-(1-\a)\log_2{x}| \geq k\sqrt{(1-\a)\log_2{x}}}}\frac{1}{t'} &\leq \sum_{\substack{t'\in T \\ |\Omega(t')-S| \geq \frac{1}{2}k\sqrt{S}}}\frac{1}{t'}\\
&\leq \sum_{|j-S| \geq \frac{k}{2}\sqrt{S}}\frac{S^j}{j!} \leq \frac{4}{k^2} \exp(S) \ll \frac{1}{k^2} (\log{x})^{1-\a}. \label{eq:treciplower2}
\end{align}
So assuming that $k$ is large, we have from \eqref{eq:treciplower1} and \eqref{eq:treciplower2} that for the final sum in \eqref{eq:lowers},
\[ \sum\frac{1}{t'} \gg (\log{x})^{1-\a}.\]
So by \eqref{eq:lowers}, we have that as $x\to\infty$,
\[ \#\mathfrak S_0 \geq \frac{x}{(\log{x})^{2+\a- \sum_{i=2}^{k}((i\log{i}+i)(\a^{i-1}-\a^i)-2\a^{i-1})+o_k(1)}}.\]
Ignoring the $o_k(1)$ term, the exponent on $\log{x}$ in the denominator simplifies under Abel summation to
\[ 2-\sum_{i=1}^{k-1} a_i \alpha^i + (k\log{k}+k) \alpha^k = 2-F(\alpha) + O((k\log k) \alpha^k). \]
Using \eqref{eq:rhodef}, we can now fix $\alpha_0 \in (1/2, \rho)$ with $F(\alpha_0) > 1-\epsilon/2$. Then if we begin the argument with $\alpha > \alpha_0$ and $k$  large enough (say $k > k_0$), we find that
\begin{equation}\label{eq:s0lower} \#\mathfrak S_0 > \frac{x}{(\log{x})^{1+\epsilon}} \end{equation}
once $x$ is large.

It remains to show that the elements of $\mathfrak S_0$ correspond to solutions $(n,m) \in \fancyB_{\phi}\times\fancyB_{\sigma}$ to $\phi(n)=\sigma(m)$. Let $\sigma \in \mathfrak S_0$, and write $\sigma$ in the form
\[ \sigma = (s_0; p_0-1, \dots, p_{k-1}-1; q_0+1, \dots, q_{k-1}+1).\]
We associate to $\sigma$ the pair $(n,m)$, where $n := p_0\cdots p_{k-1}$ and $m := q_0 \cdots q_{k-1}$.
At this point, we know that
\[ \prod{(p_i-1)} = s_0 = \prod{(q_j+1)}, \]
but we cannot conclude (yet) that $\phi(n) = \sigma(m)$, because we have not proved that $n$ and $m$ are squarefree. This is, of course, contained in showing that $(n,m)\in \fancyB_{\phi}\times\fancyB_{\sigma}$, and so we now turn to that proof. It will be enough to show that $n \in \fancyB_{\phi}$, since the proof that $m \in \fancyB_{\sigma}$ is entirely analogous.

We first establish properties (i) and (ii) in the definition of $\fancyB_{\phi}$. From \eqref{eq:Qdef2}, we have
\[ p_0 -1 \geq P^{+}(p_0-1) = Q_1 \geq x^{1/2} > v_0^{1/12} \]
and (in the notation of \eqref{eq:Qdef2})
\[ p_0-1  = s_{1,0} t' Q_1 \leq s_1 t' Q_1 \leq x = v_0. \]
Also, for $2 \leq i \leq k$, we have
\[ p_{i-1}-1 = s_{i-1,i-1} \leq v_{i-1} \]
and, in the notation used to define $\mathfrak U_2, \dots, \mathfrak U_k$,
\[ p_{i-1}-1 \geq P^{+}(p_{i-1}-1) = Q_i > v_{i-1}^{1/12}, \]
using \eqref{eq:Qdefs0}. Thus, for each $1 \leq i \leq k$,
\[ p_{i-1}-1 > v_{i-1}^{1/12} > v_i \geq p_{i}-1, \]
which shows that $p_0 > p_1 > \dots > p_{k-1}$ and so proves (i). The only statement of (ii) not shown above is that $P^{+}(p_{i-1}-1)$ is the unique prime divisor of $p_{i-1}-1$ exceeding $v_{i-1}^{1/12}$, for $2 \leq i \leq k$. In fact, any prime $p$ dividing $p_i-1$ other than $P^{+}(p_i-1)$ satisfies (in the notation of \eqref{eq:Qdefs})
\[ p \leq P^{+}(s_{i,i-1}u_i) \leq \max\{s_{i,i-1}, P^{+}(u_i)\} \leq \max\{v_i, v_{i-1}^\frac{1}{12\log_2{x}}\} = v_{i-1}^\frac{1}{12\log_2{x}} < v_{i-1}^{1/12}. \]
So we have (ii). Property (iv) follows from the definition of $\mathfrak S_k$ and the observation that each $\tau_i$ has all of its components supported on primes $> v_i \geq v_k$. To see (v), notice that the part of $\phi(n)$ supported on primes $> v_k$ can be written as the first component of $\tau_k \cdots \tau_1$, so as the $k$-fold product
\[ t_{k} t_{k-1} \cdots t_1.\]
But in our construction, the $k$ factors appearing here are squarefree and supported on pairwise disjoint sets of primes. Lastly we turn to (iii): If $i=1$, then also $j=1$, and in the notation used to define $\mathfrak U_1$, we have
\[ \Omega(p_0-1, v_1,  v_0) = \Omega(s_{1,0}t'Q_1,v_1, v_0) = \Omega(t'Q_1) = 1 + \Omega(t'); \]
the result (ii) in this case follows from (d) in the definition of $t'$. If $2 \leq i \leq k$, and $j=i$, then (in the notation used to define $\mathfrak U_2, \dots, \mathfrak U_k$)
\begin{align*} \Omega(p_{j-1}-1, v_i, v_{i-1}) &= \Omega(s_{i,i-1} t_{i,i-1}, v_i, v_{i-1}) \\&= \Omega(t_{i,i-1}) = \Omega(u_iQ_i) = 1 + \Omega(u_i),\end{align*}
and the result follows from \eqref{eq:tiprimes}. If $2 \leq i \leq k$ and $j < i$, then (in the same notation)
\[ \Omega(p_{j-1}-1,v_i, v_{i-1}) = \Omega(t_{i,j-1}), \]
and the result again follows from \eqref{eq:tiprimes}.  This completes the proof that $n \in \fancyB_{\phi}$.

Finally, notice that distinct $\sigma \in \mathfrak S_0$ induce distinct solutions $(n,m)$, by unique factorization. Thus, the number of solutions $(n,m)$ to $\phi(n)=\sigma(m)$ which we find in this way is precisely $\#\mathfrak S_0$, and the lemma follows from \eqref{eq:s0lower}.


\subsection{Proof of Lemma \ref{lem:upper}} Suppose $\alpha, k$, and $x$ are given. Take a solution $(n,n',m) \in \fancyB_{\phi}\times\fancyB_{\phi}\times\fancyB_{\sigma}$ to $\phi(n)=\phi(n')=\sigma(m)$.
Write $n = \prod_{i=0}^{k-1}p_i$, $n' = \prod_{i=0}^{k-1} p_i'$, and $m=\prod_{i=0}^{k-1}q_i$, where the $p_i$, $p_i'$, and $q_i$ are decreasing.
Put
\[ \I= \{0 \leq i \leq k-1: p_i = p_i'\}, \quad\text{so that}\quad \gcd(n,n')=\prod_{i \in \I}p_i. \]
Given $k$, there are only $O_k(1)$ possibilities for $\I$, and so we may (and do) carry out all the estimates below
assuming that $\I$
is fixed. We have
\be\label{ppq}
 (p_0-1)\cdots(p_{k-1}-1) = (p_0'-1)\cdots(p_{k-1}'-1) = (q_0+1)\cdots(q_{k-1}+1).
\ee
For each nonnegative integer $i$, let $v_i := \exp((\log{x})^{\a^i})$. We consider separately
the prime factors of each shifted prime lying in each interval $(v_i, v_{i-1}]$. For $0 \leq j \leq k-1$ and $0\leq i \leq k$, let
\[ s_{i,j}(n) := \prod_{\substack{p^a \parallel p_i-1 \\ p\leq v_i}}p^a, \quad s_{i,j}'(n) := \prod_{\substack{p^a \parallel p_i'-1 \\ p\leq v_i}}p^a,\quad s_{i,j}''(n) := \prod_{\substack{p^a \parallel q_i+1 \\ p\leq v_i}}p^a, \]
and put
\[ s_{i} := \prod_{j=0}^{k-1} s_{i,j} = \prod_{j=0}^{k-1} s_{i,j}' = \prod_{j=0}^{k-1} s_{i,j}''. \]
Also, for $0 \leq j \leq k-1$, let
\[ t_{i,j} := \frac{s_{i-1,j}}{s_{i,j}}, \quad t_{i,j}' := \frac{s_{i-1,j}'}{s_{i,j}'}, \quad t_{i,j}'' := \frac{s_{i-1,j}''}{s_{i,j}''}, \]
and put
\[ t_i := \prod_{j=0}^{k-1} t_{i,j} = \prod_{j=0}^{k-1} t_{i,j}' = \prod_{j=0}^{k-1} t_{i,j}''. \]
For each solution $(n,n',m) \in \fancyB_{\phi}\times\fancyB_{\phi}\times \fancyB_{\sigma}$ to $\phi(n)=\phi(n')=\sigma(m)$, put
\begin{align*}
 \sigma_i &:= (s_i; s_{i,0}, \dots, s_{i,k-1}; s_{i,0}', \dots, s_{i,k-1}'; s_{i,0}'', \dots, s_{i,k-1}''),  \\
 \tau_i &:= (t_i; t_{i,0}, \dots, t_{i,k-1}; t_{i,0}', \dots, t_{i,k-1}'; t_{i,0}'', \dots, t_{i,k-1}'').
\end{align*}
Note that with multiplication of $(3k+1)$-tuples defined componentwise, we have $\sigma_{i-1} = \sigma_i \tau_i.$
Let $\mathfrak S_i$ denote the set of $\sigma_i$ arising from solutions $(n,m,m') \in \fancyB_{\phi}\times \fancyB_{\phi}\times\fancyB_{\sigma}$, and let $\mathfrak T_i$ denote the corresponding set of $\tau_i$.
The number of solutions of \eqref{ppq} is
\[ \# \mathfrak{S}_0 = \sum_{\sigma \in \mathfrak S_1} \sum_{\substack{\tau \in \mathfrak{T}_1 \\ \sigma \tau \in \mathfrak S_0}} 1.\]
To estimate this quantity, we apply an iterative procedure based on the identity
\begin{equation}
\label{eq:iterative}
 \sum_{\sigma_{i-1} \in \mathfrak S_{i-1}} \frac{1}{s_{i-1}} = \sum_{\sigma_i \in \mathfrak S_i} \frac{1}{s_i} \sum_{\substack{\tau_i \in \mathfrak T_i \\ \sigma_i \tau_i \in \mathfrak S_{i-1}}} \frac{1}{t_i}.
\end{equation}

First, fix $\sigma_1 \in \mathfrak S_1$. If $\tau_1 \in \mathfrak T_1$ is such that $\sigma_1 \tau_1 \in \mathfrak S_0$, then $t_1 = t_{1,0} = t_{1,0}'=t_{1,0}'' \leq x/s_1$, $t_1$ is composed
of primes $> v_1$, and all of
\[ s_{1,0} t_1+1, \quad s_{1,0}'t_1+1, \quad s_{1,0}''t_1 -1 \]
are prime. Write $t_1 = t_1'Q$, where $Q = P^{+}(t_1)$. Then $Q = P^{+}(p_0-1) \geq x^{1/12}$ by property (ii)
in the definition of $\fancyB_{\phi}$. Hence
\[ \sum_{\substack{\tau_1 \in \mathfrak T_1 \\ \sigma_1 \tau_1 \in \mathfrak S_0}}{1} \leq \sum_{\substack{t_1' \leq x/s_1 \\ p \mid t_1' \Rightarrow p > v_1}}\sum_{x^{1/12} \leq Q\leq \frac{x}{s_1 t_1'}}{1}, \]
where the final sum is over primes $Q$ for which $s_{1,0}t_1' Q+1, s_{1,0}' t_1'Q+1$, and $s_{1,0}'' t_1' Q-1$
are also prime. By Lemma \ref{lem:sieve}, the inner sum over $Q$ is
\[ \ll
\begin{cases} \frac{x}{s_1 t_1' (\log{x})^4} (\log_2{x})^4 &\text{if $0 \not\in \I$}, \\
 \frac{x}{s_1 t_1' (\log{x})^3} (\log_2{x})^3 &\text{otherwise}.
\end{cases}
\]
Moreover,
\[ \sum_{t_1'} \frac{1}{t_1'} \leq \prod_{v_1 < p \leq x/s_1} \left(1+\frac{1}{p} + \frac{1}{p^2} + \dots\right) \ll \frac{\log{x}}{\log{v_1}} = (\log{x})^{1-\a}.\]
It follows that
\begin{equation}\label{eq:induction0} \sum_{\substack{\tau_1 \in \mathfrak T_1 \\ \sigma_1 \tau_1 \in \mathfrak S_0}} 1 \ll \frac{x}{s_1 (\log{x})^{2+\a + \chi(0)}} (\log_2{x})^4, \end{equation}
where here and below, $\chi$ denotes the characteristic function of $\I^{c} = [0,k-1]\setminus \I$.

Now suppose that $2 \leq i \leq k$, $\sigma_i \in \mathfrak S_i$, $\tau_i \in \mathfrak T_i$ and $\sigma_i \tau_i \in \mathfrak S_{i-1}$. Observe that
\begin{equation}\label{eq:threefoldfact}
 t_{i,0} \cdots t_{i,i-1} = t_{i,0}' \cdots t_{i,i-1}' = t_{i,0}'' \cdots t_{i,i-1}'' = t_{i}. \end{equation}
Let $Q_1, Q_2$, and $Q_3$ be the largest prime factors of $t_{i,i-1}, t_{i,i-1}',$ and $t_{i,i-1}''$, respectively,
and let $b, b'$, and $b''$ be the corresponding cofactors. Recall that $Q_1, Q_2, Q_3 > v_{i-1}^{1/12}$ by (ii)
in the definitions of $\fancyB_{\phi}$ and $\fancyB_{\sigma}$.  Now fix
 $\J \subset \{1,2,3\}$ indexing the distinct $Q_j$. Then $\prod_{j \in \J}Q_j \mid t_i$. Moreover, $t_i$ is squarefree  (by property (v) in the definition of $\fancyB_{\phi}$),
and so each $Q_j$ divides exactly one term from each of the three $i$-fold factorizations exhibited in \eqref{eq:threefoldfact}. Dividing all such terms
by their corresponding factor $Q_j$, we obtain an induced identity of the shape
\begin{equation}
\label{eq:threefoldfact2}
 \hat{t}_{i,0} \cdots \hat{t}_{i,i-1} = \hat{t}_{i,0}' \cdots \hat{t}_{i,i-1}' = \hat{t}_{i,0}'' \cdots \hat{t}_{i,i-1}'' = \frac{t_i}{\prod_{j \in J}Q_j},
\end{equation}
where
\begin{equation}\label{eq:trip1} \hat{t}_{i,j} = \hat{t}_{i,j}' \quad\text{for all}\quad i \in \I \cap [0,i-1], \end{equation}
and
\begin{equation}\label{eq:trip2}
\left|\Omega(\cdot, v_{i}, v_{i-1}) - (\a^{i-1}-\a^i)\log_2{x}\right| \leq 3k\sqrt{(\a^{i-1}-\a^i)\log_2{x}}
\end{equation}
for each of the $3i$ factors in the triple $i$-fold factorization \eqref{eq:threefoldfact2}.  Here we use (iii) from the definitions of $\fancyB_{\phi}$ and $\fancyB_{\sigma}$. Also, the uniqueness statement in (ii) allows us to deduce that $b = \hat{t}_{i,i-1}$, $b' = \hat{t}_{i,i-1}$, $b'' = \hat{t}_{i,i-1}$.
Putting $t= t_i/\prod_{j \in \J}{Q_j}$, we can expand
\begin{equation}
\label{eq:intersum0}  \sideset{}{^{(\J)}}\sum_{\substack{\tau_i \in \mathfrak T_i \\ \sigma_i \tau_i \in \mathfrak S_{i-1}}} \frac{1}{t_i} = \sum_{t}\frac{1}{t} \sum_{\substack{\text{$3$-fold}\\\text{factorizations}}}\sum_{\text{posns. of $Q_j$}}\prod_{j \in \J} \left(\sum_{Q_j}\frac{1}{Q_j}\right).
\end{equation}
The superscript in the left-hand sum indicates that the sum is only taken over $\tau_i$ which correspond to the index set $\J$. The second right-hand sum is over factorizations \eqref{eq:threefoldfact2} corresponding to $t$ (which necessarily satisfy \eqref{eq:trip1} and \eqref{eq:trip2}), and the third right-hand sum is over the original positions in \eqref{eq:threefoldfact} of the $Q_j$, before they were divided out to produce \eqref{eq:threefoldfact2}.

Now we estimate the innermost sum in \eqref{eq:intersum0}. Observe that
\[ p_{i-1} = s_{i,i-1} b Q_1 + 1, \quad p_{i-1}' = s_{i,i-1}' b' Q_2 + 1, \quad q_{i-1} = s_{i,i-1}'' b'' Q_3 - 1 \]
are all prime. For each $j \in \J$, let $n_j$ be the number of distinct linear forms among these which involve the prime $Q_j$. Since $Q_j$ itself is also prime,
the sieve (Lemma \ref{lem:sieve}) implies that the number of possibilities for $Q_j \leq z$ is $\ll z (\log_2{x})^{n_j+1}/(\log{z})^{n_j+1}$, and so
\[ \sum_{Q_j \geq v_{i-1}^{1/12}} \frac{1}{Q_j} \ll \frac{1}{(\log{v_{i-1}})^{n_j}} (\log_2{x})^{n_j+1} = \frac{1}{(\log{x})^{\a^{i-1} n_j}} (\log_2{x})^{n_j+1}. \]
We have $\sum_{j \in \J}n_j = 2 + \chi(i-1)$, and so
\[ \prod_{j \in \J}\left(\sum_{Q_j \geq v_{i-1}^{1/12}} \frac{1}{Q_j}\right) \ll \frac{1}{(\log{x})^{2\a^{i-1} + \a^{i-1}\chi(i-1)}}(\log_2{x})^6. \] We insert this into \eqref{eq:intersum0}. Note that the number of possible original postions of the $Q_j$ is bounded by $i^{2|\J|} \leq i^6 \ll_{k}1$. Thus, letting $h(t)$ denote the number of triple $i$-fold factorizations of the form \eqref{eq:threefoldfact2} satisfying \eqref{eq:trip1} and \eqref{eq:trip2}, we find that
\begin{equation}\sideset{}{^{(\J)}}\sum_{\substack{\tau_i \in \mathfrak T_i \\ \sigma_i \tau_i \in \mathfrak S_{i-1}}}\frac{1}{t_i} \ll_{k} \frac{1}{(\log{x})^{2\a^{i-1} + \a^{i-1}\chi(i-1)}}(\log_2{x})^6 \sum_{t}\frac{h(t)}{t}.\label{eq:tirecipup1}\end{equation}

For each value of $t$ that can appear here, we have that $t$ is squarefree, supported on primes in $(v_i, v_{i-1}]$, and satisfies
\[ |\Omega(t) - i(\a^{i-1}-\a^i)\log_2{x}| \leq 3ki \sqrt{(\a^{i-1}-\a^i)\log_2{x}}. \]
(The last inequality is a consequence of \eqref{eq:trip2}.) For all such $t$, we have for $N := \#\I\cap [0,i-1]$,
\[ h(t) \leq \sum_{\substack{i_0, \dots, i_{N}\\
i_l\text{ satisfies \eqref{eq:trip2} for $0\leq l \leq N$}\\
i_{N+1}:= \Omega(t) - \sum_{0 \leq l \leq N}{i_l}}} \binom{\Omega(t)}{i_0, \dots, i_N, i_{N+1}}((i-N)^{\Omega(t)-i_{N+1}})^2 i^{\Omega(t)};
\]
here the multinomial coefficient accounts for the common portion of the first two factorizations of \eqref{eq:threefoldfact2}, the factor $((i-N)^{\Omega(t)-i_{N+1}})^2$ bounds the number of possibilities
for the uncommon portion, and the factor $i^{\Omega(t)}$ bounds the total number of possibilities for the third factorization (which is unrestricted by \eqref{eq:trip1}).
A calculation with Stirling's formula now shows that
\[ h(t) \leq (i^{2} (i-N)^{1-N/i})^{\Omega(t)} \exp(O_k(\log_3{x}\sqrt{\log_2{x}})), \]
uniformly in $t$. Put
\[ I:= i(\a^{i-1}-\a^i)\log_2{x} + 3k^2 \sqrt{(\a^{i-1}-\a^i)\log_2{x}}. \]
Then with
\[
S:= \sum_{v_i < p \leq v_{i-1}} \frac{1}{p} = (\a^{i-1}-\a^i)\log_2{x} + O(1),
\]
the multinomial theorem gives us that
\begin{equation}\label{eq:tirecipup2} \sum_{t} \frac{h(t)}{t} \leq \exp(O_k(\log_3{x}\sqrt{\log_2{x}})) \sum_{j \leq I} (i^2 (i-N)^{1-N/i})^{j} \frac{S^j}{j!}.\end{equation}
For large $x$, we have $I < i^2 S \leq i^2 S (i-N)^{1-N/i}$, and so by Lemma \ref{lem:poisson1},
\begin{align}
 \sum_{j \leq I} (i^2 (i-N)^{1-N/i})^{j} \frac{S^j}{j!} &\leq \left(\frac{e i^2 (i-N)^{1-N/i} S}{I}\right)^{I} \notag\\
 &\leq \exp(O_k(\sqrt{\log_2{x}})) (ei(i-N)^{1-N/i})^{I} \notag\\
 &= (\log{x})^{(\a^{i-1}-\a^i)(i+i\log{i}) +(i-N)\log(i-N)(\a^{i-1}-\a^i) + o_k(1)}.\label{eq:tirecipup3}
\end{align}
From \eqref{eq:tirecipup1}, \eqref{eq:tirecipup2}, and \eqref{eq:tirecipup3}, we deduce that
\begin{equation}\label{eq:inductiongen} \sum_{\substack{\tau_i \in\mathfrak T_i \\ \sigma_i \tau_i \in \mathfrak S_{i-1}}}\frac{1}{t_i} \leq (\log{x})^{(\a^{i-1}-\a^i)(i+i\log{i}) -2\a^{i-1} +(i-N)\log(i-N)(\a^{i-1}-\a^i)-\chi(i-1) \a^{i-1}+ o_k(1)};
\end{equation}
we use here that there are only finitely many possibilities for $\J$, so that we can drop the superscript $(\J)$ on the sum.

By \eqref{eq:iterative}, \eqref{eq:induction0}, and \eqref{eq:inductiongen}, we see that
\begin{multline*} \#\mathfrak S_0 \leq (\log{x})^{o_k(1)} \frac{x}{(\log{x})^{2+\a - \sum_{i=2}^{k}((\a^{i-1}-\a^i)(i+i\log{i})-2\a^{i-1})}} \\
\times (\log{x})^{\sum_{i=2}^{k}(i-\#\I\cap[0,i-1])\log(i-\#\I\cap[0,i-1])(\a^{i-1}-\a^i) -\sum_{i=0}^{k-1}\chi(i) \a^i} \times  \sum_{\sigma_k \in \mathfrak S_k}\frac{1}{s_k},
\end{multline*}
with the convention that $0\log{0}=0$. From (iv) in the definitions of $\fancyB_{\phi}$ and $\fancyB_{\sigma}$, the final sum is $6^{-k} \le 1$.
Also,
\begin{align*}
2+\a - \sum_{i=2}^{k}((\a^{i-1}-\a^i)(i+i\log{i})-2\a^{i-1})  &= 2 - \sum_{i=1}^{k-1} a_i \alpha^i + (k\log{k}+k)\alpha^k \\
&= 2-F(\alpha) + O((k\log{k}) \alpha^k).
\end{align*}
Since $F(\a) \le F(\rho)=1$, for sufficiently large $k$, this last expression is at least $1-\eps/2$.
So the desired upper bound on $\mathfrak S_0$ follows if it is shown that
\begin{equation}\label{eq:nonpositive}
 \sum_{i=2}^{k}(i-\#\I\cap[0,i-1])\log(i-\#\I\cap[0,i-1])(\a^{i-1}-\a^i) \leq \sum_{i=0}^{k-1}\chi(i) \a^i. \end{equation}
For brevity, write $M(i):= i - \#\I\cap[0,i-1] = \#\I^c\cap [0,i-1]$. Abel summation implies that for the left-hand side of \eqref{eq:nonpositive}, we have
\[ \sum_{i=2}^{k} M(i)\log{M(i)} (\a^{i-1}-\a^i) \leq \sum_{i=1}^{k-1} (M(i+1)\log{M(i+1)}-M(i)\log{M(i)})\a^i. \]
Suppose that
\[ 1 \leq i_1 < i_2 < \dots < i_L \leq k-1 \]
is a list of the elements of $\I^c$ in $[1,k-1]$. If $0 \not\in \I^{c}$, then
\begin{multline*} \sum_{i=1}^{k-1} (M(i+1)\log{M(i+1)}-M(i)\log{M(i)})\a^i\\  = \sum_{l=2}^{L} (a_{l-1}+1)\a^{i_l}
= \sum_{l=2}^{L} a_{l-1} \a^{i_l} + \sum_{i=i_2}^{k-1} \chi(i)\a^i \\\leq
\alpha^{i_{1}}\sum_{i=1}^{\infty} a_i \a^i +
\sum_{i=i_2}^{k-1} \chi(i)\a^i \leq \a^{i_1} + \sum_{i=i_2}^{k-1} \chi(i)\a^i =
\sum_{i=0}^{k-1}\chi(i)\a^i.
\end{multline*}
If $0 \in \I^{c}$, then
\begin{multline*}
 \sum_{i=1}^{k-1} (M(i+1)\log{M(i+1)}-M(i)\log{M(i)})\a^i \leq \sum_{l=1}^{L} (a_l+1)\a^{i_l} \\ \leq \sum_{i=1}^{L} a_l \a^l + \sum_{i=1}^{k-1} \chi(i)\a^i < 1 + \sum_{i=1}^{k-1} \chi(i)\a^i = \sum_{i=0}^{k-1}\chi(i) \a^i.
\end{multline*}
So \eqref{eq:nonpositive} holds in either case.

This concludes the proof of the first half of Lemma \ref{lem:upper}. The estimate for the number of solutions to $\phi(n)=\sigma(m)=\sigma(m')$, where $(n,m,m') \in \fancyB_{\phi}\times\fancyB_{\sigma}\times\fancyB_{\sigma}$, is entirely analogous.

\medskip
{\bf Remarks.}  The term $(\log x)^{1+o(1)}$ in Theorem 1 can be sharpened
by allowing $k,\a$ to depend on $x$ in the above argument, or by using the
fine structure theory of totients from \cite{ford98}.

\providecommand{\bysame}{\leavevmode\hbox to3em{\hrulefill}\thinspace}
\providecommand{\MR}{\relax\ifhmode\unskip\space\fi MR }
\providecommand{\MRhref}[2]{%
  \href{http://www.ams.org/mathscinet-getitem?mr=#1}{#2}
}
\providecommand{\href}[2]{#2}

\end{document}